\documentclass[10pt, a4paper]{amsart}
\usepackage{amssymb, amsmath, amsfonts, amsthm, verbatim}
\usepackage{hyperref}
\newtheorem{theorem}{Theorem}[section]
\newtheorem{lemma}[theorem]{Lemma}
\newtheorem{definition}[theorem]{Definition}

\newtheorem{q}[theorem]{Question}
\newtheorem{p}[theorem]{Problem}

\title{A new infinite family of non-abelian strongly real Beauville $p$-groups for every odd prime $p$}

\author{Ben Fairbairn}
\address{Ben Fairbairn, Department of Economics, Mathematics and Statistics, Birkbeck, University of London, Malet Street, London WC1E 7HX, United Kingdom}
\email{b.fairbairn@bbk.ac.uk}

\begin{document}
\maketitle

\begin{abstract}
We prove that there exist infinitely many a non-abelian strongly real Beauville $p$-group for every prime $p$. Previously only finitely many in the case $p=2$ have been constructed.
\end{abstract}

\section{Introduction}\label{Sec1}
We begin with our main definition.

\begin{definition}\label{MainDef}
Let $G$ be a finite group and for $g,h\in G$ let
\[
\Sigma(g,h):=\bigcup_{i=1}^{|G|}\bigcup_{k\in G}\{(g^i)^k,(h^i)^k,((gh)^i)^k\}.
\]

A set of elements $\{\{x_1,y_1\},\{x_2,y_2\}\}\subset G\times G$ is a
\emph{Beauville structure of} $G$ if and only if $\langle
x_1,y_1\rangle=\langle x_2,y_2\rangle=G$ and
\begin{equation}
\Sigma(x_1,y_1)\cap\Sigma(x_2,y_2)=\{e\}.\tag{$\dagger$}
\end{equation}

If $G$ has a
Beauville structure then we call $G$ a \emph{Beauville
group}.

 Let $G$ be a Beauville group and let $X =\{\{x_1, y_1\},\{x_2, y_2\}\}$ be a
Beauville structure for $G$. We say that $G$ and $X$ are \emph{strongly real} if there exists an
automorphism $\phi\in\mbox{Aut}(G)$ and elements $g_i\in G$ for $i = 1, 2$
such that
\begin{equation}
g_i\phi(x_i)g_i^{-1}=x_i^{-1}\mbox{ and }g_i\phi(y_i)g_i^{-1}=y_i^{-1}\tag{$\star$}
\end{equation}
for $i=1,2$.
\end{definition}

Beauville groups were originally introduced in connection with a class of complex surfaces of general type known as Beauville surfaces. These surfaces posses many useful geometric properties: their automorphism groups and fundamental groups are relatively easy to compute and these surfaces are rigid in the sense of admitting no non-trivial deformations and are thus isolated points in the moduli space of surfaces of general type. They provide cheap counterexamples and are a useful testing ground for conjectures. What makes these surfaces so easy to work with is the fact that doing so boils down to working inside the corresponding Beauville group and structure.

For $p\geq5$ abelian strongly real Beauville $p$-groups are easy to construct: writing $C_n$ for the cyclic group of order $n$, in \cite[Section 3]{C} Catanese classified the abelian Beauville groups proving that they are precisely the groups $C_n\times C_n$ where $n>1$ is coprime to 6. Consequently, if $p\geq5$ we can construct infinitely many abelian strongly real Beauville $p$-groups by simply setting $n$ to be a power of $p$ (though Catanese's result also tells us that there are no abelian Beauville 2-groups or 3-groups.) Non-abelian examples are much harder to construct. As far as the author is aware the only previously published examples of non-abelian strongly real Beauville $p$-groups is a pair of 2-groups constructed by the author in \cite[Section 7]{MoreF}. There is sound reason for believing the case of $p$ odd is harder. In \cite{HM} Helleloid and Martin prove that automorphism group of a finite $p$-group is almost always a $p$-group. In particular, if $p$ is odd, then typically no automorphism like the $\phi$ in condition ($\star$) of Definition \ref{MainDef} exists since such an automorphism must necessarily have even order. Even ignoring the strongly real condition, Beauville $p$-groups are more difficult to construct in general --- a commonly used trick for showing that condition $(\dagger)$ of Definition \ref{MainDef} is satisfied is to find a Beauville structure such that $o(x_1)o(y_1)o(x_1y_1)$ is coprime to $o(x_2)o(y_2)o(x_2y_2)$ but this clearly cannot be done in a $p$-group since every non-trivial element has an order that's a power of $p$. Further motivation comes from the fact that in some sense `most' finite groups are $p$-groups \cite{obrian1,obrian2} and thus establishing the general picture in this case a long way to establishing the wider picture in general. Despite the above difficulties, a number of authors have found a variety of ingenious constructions for them \cite{BBF,BBPV1,BBPV2,MoreF,new,FGJ,Gul1,JW,SV}. Our main result is as follows.

\begin{theorem}\label{MainThm}
Let $p$ be an odd prime and let $q$ and $r$ be powers of $p$. If $q$ and $r$ are sufficiently large, then groups $C_q\wr C_r/Z(C_q\wr C_r)$ are strongly real Beauville groups.
\end{theorem}

Whilst this document was in preparation, in \cite{Gul2} G\"{u}l gave the first construction of infinite families of strongly real Beauville $p$-groups for $p$ odd. Our result relies on a more general construction and unlike the one given in \cite{Gul2} there are infinitely many orders of $p$-groups for which our construction gives multiple groups of the same order. For example when $(q,r)=(3^{28},3^3)$ or $(q,r)=(3^3,3^5)$ we obtain groups of order $3^{731}$ which cannot be isomorphic since they have centers of different orders.


For general surveys on these and related matters see \cite{BSurvey,FSurvey,JSurvey,S,WSurvey}.


\section{Proof of the Main Theorem}

\subsection{A General Lemma}

The following general lemma is straightforward to prove but useful.

\begin{lemma}\label{MainLem}
Let $G$ be a finite group; let $Z\leq G$ be a characteristic subgroup; let $t\in Aut(G)$ and let $x_1,y_1,x_2,y_2\in G$ have the properties that
\begin{equation}
\Sigma(x_1,y_1)\cap\Sigma(x_2,y_2)\subseteq Z,\tag{$\dagger\dagger$}
\end{equation}
$\langle x_1,y_1\rangle=\langle x_2,y_2\rangle=G$ and
\begin{equation}
x_i^t=x_i^{-1}\mbox{ and }y_i^t=y_i^{-1}\mbox{ for }i=1,2.\tag{$\star\star$}
\end{equation}
Then $G/Z$ is a strongly real Beauville group.
\end{lemma}

\begin{proof}
By hypothesis $\langle x_i,y_i\rangle=G$ for $i=1,2$ and by condition ($\star\star$) we have that $\Sigma(x_1Z,y_1Z)\cap\Sigma(x_2Z,y_2Z)=\{e\}$. Now $t$ induces an automorphism of $G/Z$ that can be used to define a strongly real Beauville structure since for $i=1,2$

\begin{tabular}{rclr}
$(x_iZ)^t$&=&$x_i^tZ^t$&\\
&=&$x_i^tZ$&[since $Z$ is a characteristic subgroup, by hypothesis]\\
&=&$x_i^{-1}Z$&[since $x_i^t=x_i^{-1}$, by hypothesis]\\
&=&$(x_iZ)^{-1}$&
\end{tabular}

\noindent and similarly $(y_iZ)^t=(y_iZ)^{-1}$. It follows that $\{\{x_1Z,y_1Z\},\{x_2Z,y_2Z\}\}\subset G/Z$ gives a strongly real Beauville structure since homomorphic images of generating sets are generating sets.
\end{proof}

\subsection{The Groups}\label{TheGroups}

By Lemma \ref{MainLem} the problem of showing that $G/Z$ is a strongly real Beauville group can be lifted to the potentially much simpler task of working inside the group $G$ instead.

Let $p$ be a prime and let $q$ and $r$ be powers of $p$. We will consider the wreath product $C_q\wr C_r$. Intuitively, this is a natural class of groups to consider since having a large abelian subgroup, the subgroup isomorphic to $C_q^r$, ensures that most elements have a large centralizer and so conjugacy classes are small thus making it easier to satisfy condition ($\dagger$) of Definition \ref{MainDef}. Unfortunately, wreath products like these can never be Beauville groups --- for any generating pair $\{g,h\}$ we have that $\Sigma(g,h)$ contains the (non-trivial) center of the group making it impossible to satisfy condition ($\dagger$) of Definition \ref{MainDef}. The above lemma, however, enables us to work around this problem.

To give explicit names to elements of these groups we will define our copy of $C_q\wr C_r$ by the following presentation
\[
\bigg\langle x,y\,\bigg|\,x^q,y^r,[x,x^{y^i}]\mbox{ for }i=1,\ldots,\frac{r-1}{2}\bigg\rangle.
\]

\subsection{A Representation}\label{RepSec}

To help show that condition ($\star\star$) of Lemma \ref{MainLem} is satisfied by the elements we will be considering we will calculate the traces of matrices representing the elements of our groups. To do this we first consider a permutation representation on the points $\{1,\ldots,qr\}$.

For the element $x$ we take the permutation
\[
\Xi:=\bigg(\frac{r+1}{2},\frac{r+1}{2}+r,\frac{r+1}{2}+2r,\ldots,\frac{r+1}{2}+(q-1)r\bigg)
\]
and for the element $y$ we take the permutation
\[
\Upsilon:=(1,2,\ldots,r)(r+1,r+2,\ldots,2r)\cdots(qr-r+1,qr-r+2,\ldots,qr)
\]
so that $\Xi$ cyclicly permutes the midpoints of the cycles defining $\Upsilon$.

To construct the element $\phi$ in Definition \ref{MainDef} we consider the involution
\[
t:=\bigg(1,qr\bigg)\bigg(2,qr-1\bigg)\cdots\bigg(\frac{qr+1}{2}-1,\frac{qr+1}{2}+1\bigg).
\]
Direct calculation shows that that given the $x$ and $y$ above we have that $x^t=x^{-1}$ and $y^t=y^{-1}$. Moreover, $t$ is an automorphism of the group since it simply sends each of the defining relations in the above presentation to to one that is immediately implied by them.

Next we construct a degree $qr+2$ representation of these elements which is denoted in Figure 1. This representation is given as $W\oplus V_1\oplus V_2$ where $W$ is acted on by the permutation matrices defined by $\Xi$ and $\Upsilon$ whilst $V_1$ and $V_2$ are linear representations included to give the matrices traces that will distinguish the various elements we are interested in.

\begin{figure}
\[
x\mapsto X:=\left(\begin{array}{c|c|c}
\Xi&&\\
\hline
&\zeta_q&\\
\hline
&&1
\end{array}\right)\mbox{, }
y\mapsto Y:=\left(\begin{array}{c|c|c}
\Upsilon&&\\
\hline
&1&\\
\hline
&&\zeta_r
\end{array}\right)
\]
\label{MainFig}\caption{The representation. Here we write permutations to denote their corresponding permutation matrix; 1 to denote the $1\times1$ identity matrix and $\zeta_n$ denotes a primitive $n^{th}$ root of unity.}
\end{figure}

\subsection{The Beauville Structure}

In this subsection we finally give the Beauville structure that proves Theorem \ref{MainThm}.

We claim that the elements $x$, $y$, $x^yxx^{y^{-1}}$ and $xyx$ satisfy the hypotheses of Lemma \ref{MainLem} and thus $C_q\wr C_r/Z(C_q\wr C_r)$ is a strongly real Beauville group since the center of a group is always a characteristic subgroup.

We saw in Section \ref{TheGroups} that the automorphism $\phi$ defined by conjugation by the element $t$ satisfies condition ($\star\star$).

Writing $Tr(A)$ to denote the trace of the matrix $A$ we note that for non-identity powers of the matrices
\[
Tr(X^i)=qr-q+1+\zeta_q^i\mbox{, }Tr(Y^i)=1+\zeta_r^i\mbox{ and }Tr((XY)^i)=\zeta_q^i+\zeta_r^i
\]
and
\[
Tr((X^YXX^{Y^{-1}})^i)=qr-3q+1+\zeta_q^{3i}\mbox{, }Tr((XYX)^i)=\zeta_q^{2i}+\zeta_r^i
\]
\[
\mbox{ and }Tr((X^YXX^{Y^{-1}}XYX)^i)=\zeta_q^{5i}+\zeta_r^i.
\]
Since all the non-central powers of these all have distinct traces it follows that no non-trivial power of $x$, $y$ and $xy$ can be conjugate to any non-trivial power of $xyx$, $x^yxx^{y^{-1}}$ and $xyxx^yxx^{y^{-1}}$, aside from the powers of them lie in the center. It follows that these elements satisfy condition ($\dagger\dagger$) of Lemma \ref{MainLem}.

Finally it only remains to prove that our elements generate the whole group. We defined $x$ and $y$ so that $\langle x,y\rangle=C_q\wr C_r$ so it only remains to show that $xyx$ and $x^yxx^{y^{-1}}$ together generate the group. To do this note that it is sufficient to express $x$ as a word in $xyx$ and $x^yxx^{y^{-1}}$ since once we have $x$ we have that
\[
x^{-1}(xyx)x^{-1}=(x^{-1}x)y(xx^{-1})=y.
\]
First suppose that $p\not=3$. Note that $x^yxx^{y^{-1}}((x^yxx^{y^{-1}})^{-1})^{xyx}=x^{y^{-1}}x^{-y^2}$ and conjugating this by $xyx$ gives us $xx^{-y^3}$. Conjugating this by $(xyx)^3$ gives us $x^{y^3}x^{-y^6}$ and so we have $xx^{-y^3}(x^{y^3}x^{-y^6})^{-1}=xx^{-y^6}$. Since 3 is coprime to the order of $y$ we can easily repeat this to obtain $xx^{-y^i}$ for any $i$. In particular we have the elements $xx^{-y}$ and $xx^{-y^{-1}}$ and so
$(x^yxx^{y^{-1}})(xx^{-y})(xx^{-y^{-1}})=x^3$. Since 3 is coprime to the order of $x$ we can power this up to finally obtain $x$.

Note that if $r=3$ the above will not work since $x^yxx^{y^{-1}}$ generates the center of the group in this case and so $\langle xyx, x^yxx^{y^{-1}}\rangle$ is abelian and thus not the whole group. We thus insist that $r>3$ when $p=3$. (It is natural to consider such a restriction since the group $C_3\wr C_3$ is known to be too small to be a Beauville group, let alone a strongly real one and so the same is obviously true of any quotient of it --- see \cite[Corollary 1.9]{BBF}) In this case it turns out to more convenient to use $x^{y^2}x^yxx^{y^{-1}}x^{y^{-2}}$ instead of $x^yxx^{y^{-1}}$. Whilst the trace calculations for this element will clearly be a little different in this case, they nonetheless show that the only powers of this element and its product with $xyx$ that are conjugate to any of $x$, $y$ or $xy$ lie in the center so it is a perfectly valid candidate. Moreover it is also clearly inverted by our automorphism. To prove that this pair generates an argument entirely analogous to that of the previous paragraph can be carried out going via the elements $xx^{-y^5}$ which works since 5 is coprime to 3.

\section{Concluding Remarks}

Here we have explicitly constructed some families of strongly real Beauville $p$-groups, however it would be of interest to have a general test for being a strongly real Beauville $p$-group, akin to the general conditions for a $p$-group being a Beauville group given by Fern\'{a}ndez-Alcober and G\"{u}l in \cite{new} or by Jones and Wolfart in \cite[Theorem 11.2]{JW}.

We also remark that the special case of $q=r=p\geq5$ was first shown to be Beauville groups that are not necessarily strongly real by Jones and Wolfart in \cite[Exercise 11.1]{JW}.

Whilst,hilst infinitely many strongly real Beauville $p$-groups are now known to exist for each odd prime it would be interesting to know how frequently they occur.

\begin{q}
\begin{enumerate}
\item[(a)] How does the proportion of Beauville groups of order $p^n$ that are strongly real vary as $n$ increases?
\item[(b)] How does the proportion of Beauville groups of order $p^n$ that are strongly real vary as $p$ increases?
\end{enumerate}
\end{q}

Finally a related question about the proportion of $p$-groups that are Beauville are posed in \cite[Question 1.8]{BBF}.

An even more obvious and pressing problem is the following:

\begin{p}
Construct infinitely many strongly real Beauville 2-groups.

\end{p}

\end{document}